\documentclass[11pt]{article}

\usepackage{amsmath,amssymb,amsthm,mathtools,setspace,graphicx,array,tikz,url}
\usepackage[text={6.5in,8.4in}, top=1.3in]{geometry}
\newtheorem{theorem}{Theorem}[section]
\newtheorem{proposition}[theorem]{Proposition}
\newtheorem{lemma}[theorem]{Lemma}
\newtheorem{corollary}[theorem]{Corollary}
\newtheorem{conjecture}[theorem]{Conjecture}

\begin{document}
	
\title{
	On the minimum density of monotone subwords
}
	
\author{
	Raphael Yuster
	\thanks{Department of Mathematics, University of Haifa, Haifa 3498838, Israel. Email: raphael.yuster@gmail.com\,.}
}
	
\date{}
	
\maketitle
	
\setcounter{page}{1}
	
\begin{abstract}
	We consider the asymptotic minimum density $f(s,k)$ of monotone $k$-subwords of words over a totally ordered alphabet of size $s$. The unrestricted alphabet case, $f(\infty,k)$, is well-studied, known for $f(\infty,3)$ and $f(\infty,4)$, and, in particular, conjectured to be rational for all $k$. Here we determine $f(2,k)$ for all $k$ and determine $f(3,3)$, which is already irrational. 
	We describe an explicit construction for all $s$ which is conjectured to yield $f(s,3)$.
	Using our construction and flag algebra, we determine $f(4,3),f(5,3),f(6,3)$ up to $10^{-3}$
	yet argue that flag algebra, regardless of computational power, cannot determine $f(5,3)$ precisely.
	Finally, we prove that for every fixed $k \ge 3$, the gap between $f(s,k)$ and  $f(\infty,k)$ is $\Theta(\frac{1}{s})$.

\end{abstract}

\section{Introduction}

Let ${\mathcal S}$ be a totally ordered set. An {\em $n$-word} over the {\em alphabet} ${\mathcal S}$ is a sequence $w : [n] \rightarrow {\mathcal S}$. The sequence elements are called {\em letters}.
A subsequence (hereafter, {\em subword}) of a word is {\em monotone} if it is monotone non-increasing or monotone non-decreasing.

The study of patterns (such as monotone patterns) in words is a classical problem in combinatorics;
see \cite{HM-2009} for a survey of central results in the area. The special case of {\em permutations} which, in the realm of patterns, is the case where all letters in the word are distinct, dates back to the classical theorem of Erd\H{o}s and Szekeres \cite{ES-1935}, who proved that a permutation of length $(k-1)^2+1$ has a monotone $k$-length subword. Hence, one expects that as $n$ grows, the number of monotone $k$-subwords grows accordingly. Nevertheless, it is usually difficult to determine the minimum amount of monotone patterns.
In fact, it was not long ago that Samotij and Sudakov \cite{SS-2015} determined this minimum for permutations in the special case where $n \le k^2+ck^{3/2}/\log k$ for some absolute constant $c$ and $k$ sufficiently large.
As pattern quantification problems are mostly hard to enumerate precisely, it is of interest to study asymptotic density, whether minimum density (e.g., \cite{BHL+-2015}), maximum density (e.g., \cite{SS-2018}) or the random word setting (e.g., \cite{CR-2016}). As we are interested in monotone subwords, we consider their
minimum asymptotic density, as formally defined below.

Hereafter, we shall assume that ${\mathcal S}$ is finite, and set, without loss of generality, ${\mathcal S}=\Sigma_s :=\{0,\ldots,s-1\}$. Let $m(k,w)$ denote the number of monotone $k$-subwords of a word $w$ and let $f(s,k,n)$ denote the minimum of $m(k,w)/\binom{n}{k}$ taken over all $w \in (\Sigma_s)^n$.
Notice that $g(k,n) := f(n,k,n)$ is the minimum density of monotone $k$-subwords of a permutation of $[n]$.
Let $f(s,k)$ (respectively, $g(k)$) denote the limit of $f(s,k,n)$ (respectively, $g(k,n)$) as $n$ goes to infinity. The limit clearly exists as the sequences $f(s,k,n)$ and $g(k,n)$ are monotone non-decreasing with $n$.
It is easy to see that for every fixed $k$, $f(s,k)$ is a non-increasing function of $s$ and its limit is $g(k)$.

The problem of determining $g(k,n)$ was initiated by Atkinson, Albert and Holton (see \cite{myers-2002}).
Myers described a construction upper-bounding $g(k,n)$ for every $k,n$ and conjectured that it yields
the exact bound, and hence the exact value of $g(k)$. Myers proved his conjecture for $k=3$,
determining $g(3,n)$ for every $n$. He also gave a simple proof that $g(3)=\frac{1}{4}$, based on
Goodman's classical formula \cite{goodman-1959}. As for the general case, Myers' conjecture is that $g(k)=\frac{1}{(k-1)^{k-1}}$. This value was shown to hold for the special class of {\em layered permutations} by de Oliveira Bastos, and Coregliano \cite{OC-2016}. The case $g(4)=\frac{1}{27}$ was proved
by Balogh, Hu, Lidick{\`y}, Pikhurko, Udvari, and Volec \cite{BHL+-2015} with a sophisticated use of flag algebra, but otherwise the general conjecture of Myers regarding $g(k)$ is wide open. Here we address $f(s,k)$ and, as we shall see, the problem becomes
challenging even for $k=3$ and even for small size alphabets, unlike the case $k=3$ of permutations 
which, recall, is fairly simple.

Let us start with the smallest size nontrivial alphabet, $s=2$. In this case we can determine
$f(2,k)$ for all $k$, but even this requires some effort, as the proof of the theorem below suggests.
\begin{theorem}\label{t:exact-2}
	$f(2,k) = \frac{k}{2^{k-1}}$.
\end{theorem}
As can be seen from the proof of Theorem \ref{t:exact-2}, one can determine $f(2,k,n)$ exactly for all $n$;
the minimum is attained by an alternating binary word (and, when $n$ is odd, {\em only} by an alternating
binary word).

We next turn to larger alphabets. In the next two theorems we consider the case $k=3$ (monotone triples),
for which we can get precise and almost precise results for small $s$, and conjecture the precise result
for every $s$. Our next theorem determines the smallest case, $f(3,3)$ which turns out to be irrational (recall that in the permutation case, all conjectured values are rationals).
\begin{theorem}\label{t:exact-3}
	$f(3,3) = 2-\sqrt{2}=0.5857...$.
\end{theorem}

In Section 3, we present, for every $s$, an explicit construction of a sequence $S_s= \{w^n\}_{n=1}^{\infty}$ of words over $\Sigma_s$ where $w^n$ is an $n$-word, which is conjectured to be asymptotically optimal, in the sense that
$m(3,w^n)/\binom{n}{3}$ is conjectured to approach $f(s,3)$ as $n$ goes to infinity.
Our construction is accompanied with an explicitly defined multivariate degree $3$ polynomial $h_s$ with
$\lfloor (s-1)/2 \rfloor$ variables.
Let $q(s)$ denote the minimum of $h_s$ in the positive orthant of the closed halfspace defined by summing the variables to $\frac{1}{2}$. It is then conjectured (see Conjecture \ref{conj:1})
that $q(s)=f(s,3)$ and proved that $q(s) \ge f(s,3)$ ($q(3)$ coincides with $f(3,3)$ of Theorem \ref{t:exact-3}).
While we can obviously compute $q(s)$ to arbitrary precision for every $s$, we have done so
analytically for all $4 \le s \le 7$ as can be seen from the upper bound column of Table \ref{table:bounds}.
But how plausible is our conjecture in the sense of a supporting lower bound? It is exact for $s=3$, and we prove that it is {\em very close} to $f(s,3)$ for $s \in \{4,5,6\}$. To this end, we shall use the flag algebra framework of Razborov \cite{razborov-2007}. Our approach follows along the lines of \cite{BHL+-2015}, with two major differences: The flag algebraic objects that we work with are, what we call, {\em word graphs},
which are certain edge colorings of complete graphs that capture monotonicity in words. This is somewhat different than \cite{BHL+-2015} whose objects are permutation graphs. While a single word graph may
correspond to many words (hence, considerably reducing the search space instead of working directly with words),
there are still many distinct word graphs for small $s$ and $n$, as opposed to permutation graphs.
For example, while there are only $776$ permutation graphs for $n=7$, there are $12712$ word graphs
already for $n=7$ and $s=6$ (see Table \ref{table:gsl}). This makes our semidefinite programs rather large.
On the other hand, while \cite{BHL+-2015} obtain an exact rational result, we only claim a lower bound - nor can we ask for an exact result yielding $q(s)$ using a 
flag algebra generated sdp - because $q(s)$ is an integer polynomial evaluated at a real algebraic root of another high degree polynomial (in fact $q(5)$ is determined by an algebraic number of degree $4$, see Subsection \ref{subsec:upper}).
We thus argue that, regardless of computational power, flag algebra just by itself cannot be used to
determine $f(s,5)$. 
Nevertheless, using flag algebra, we obtain the lower bounds for $f(s,3)$ when $s \in \{4,5,6\}$ given in the lower bound column of Table \ref{table:bounds}. As can be seen, these lower bounds
differ from $q(s)$ (hence $f(s,3)$) by at most $10^{-3}$. Our findings are summarized in the following theorem.
\begin{theorem}\label{t:almost-exact}
	The following holds for $f(s,3)$ where $4 \le s \le 7$.
	\begin{table}[h]
		\centering
		\begin{tabular}{c||c|c|c}
			$s$ &  upper bound (conjectured tight) & lower bound & gap smaller than\\
			\hline
			$4$ & $0.5133...$ & $0.5123...$ & $0.001$\\
			\hline
			$5$ & $0.4610...$ & $0.4604...$ & $0.001$\\
			\hline
			$6$ & $0.4288...$ & $0.4280...$ & $0.001$\\
			\hline
			$7$ & $0.4033...$ & &\\
		\end{tabular}
		\caption{Upper and lower bounds for $f(s,3)$. All upper bounds are conjectured to be optimal.}
		\label{table:bounds} 
	\end{table} 
\end{theorem}

Finally, we consider the general case $f(s,k)$. Recall that $f(s,k)$ is a non-increasing function of $s$ and its limit is $g(k)$; yet how quickly does it converge as a function of $s$? The following theorem provides the answer.
\begin{theorem}\label{t:general}
	$f(s,k)-g(k)=\Theta(\frac{1}{s})$.
	In particular, $f(s,k) = \frac{1}{(k-1)^{k-1}} + \Theta(\frac{1}{s})$ for $k=3,4$ and, assuming the conjecture of Myers,  $f(s,k) = \frac{1}{(k-1)^{k-1}} + \Theta(\frac{1}{s})$ holds for all $k$.
\end{theorem}

The remainder of this paper is organized as follows. In Section \ref{sec:exact} we prove our exact results, Theorems \ref{t:exact-2} and \ref{t:exact-3}. The case $k=3$ consisting of the conjectured optimal upper bound for $f(s,3)$ and the explicit construction of the polynomial $h_s$ is given in Section \ref{sec:construction}.
Section \ref{sec:flag} describes our proof of the lower bounds using flag algebra; these two sections, together with Appendix \ref{appendix:A}, yield Theorem \ref{t:almost-exact}.
In the final Section \ref{sec:converge} we prove Theorem \ref{t:general}.

\section{Exact results}\label{sec:exact}
In this section we prove our exact results starting with the case of the binary alphabet, Theorem \ref{t:exact-2}, and proceeding with the case of the ternary alphabet, Theorem \ref{t:exact-3}.
We require some notation: For a word $w$, let $w_t$ denote the letter of $w$ at position $t$
(so that $w_1$ is the first letter of $w$),
let $q^<(\ell,t,w)$ denote the number of occurrences of the letter $\ell$ at positions of $w$ smaller
than $t$ and let $q^>(\ell,t,w)$ denote the number of occurrences of the letter $\ell$ at positions of $w$
larger than $t$.

\subsection{Binary alphabet}

\begin{proof}[proof of Theorem \ref{t:exact-2}]
We assume that all words are over $\Sigma_2=\{0,1\}$.
It is immediate to verify that for every $k$-word $u$,
the probability that a randomly chosen $k$-subword of the alternating $n$-word $01010...$ equals $u$ is $\frac{1}{2^k}+o_n(1)$.
As there are precisely $2k$ monotone $k$-words over $\Sigma_2$, it follows that $f(2,k) \le \frac{2k}{2^k}$.

To prove the lower bound, we show that binary words of length $n$ minimizing the density of monotone $k$-subwords
must be close to an alternating $n$-word. In fact, we will prove that if $n$ is odd, then it {\em must} be an
alternating $n$-word. Note that it suffices to consider the case where $n$ is odd to obtain the claimed lower bound on $f(2,k)$. So, let $n$ be odd and let $w$ be a word of length $n$ with $m(k,w) \le m(k,w^*)$ for all $w^* \in (\Sigma_2)^n$. Without loss of generality, assume that the number of $0$'s in $w$ is larger than the number of $1$'s. We may further assume that $w$ is not alternating, otherwise there is nothing to prove.
Let $x$ denote the number of $0$'s in $w$ and let $y=n-x < x$ denote the number of $1$'s in $w$.

We first claim that $w$ must start and end with $0$.
Assume, to the contrary, that it starts with $1$ and let $t$ be the smallest position such that $w_t=0$, so we have $t \ge 2$.
Flip $w_t$ and $w_{t-1}$ to obtain the word $w^*$. We have
\begin{align*}
	m(k,w)-m(k,w^*) & = \sum_{h=0}^{k-2}\binom{q^<(1,t-1,w)}{h}\binom{q^>(0,t,w)}{k-2-h}
	\! - \! \binom{q^<(0,t-1,w)}{h}\binom{q^>(1,t,w)}{k-2-h}\\
	& =  \left(\sum_{h=0}^{k-2} \binom{t-2}{h}\binom{x-1}{k-2-h}\right) - \binom{y-t+1}{k-2}\\
	& \ge  \binom{x-1}{k-2} - \binom{y-t+1}{k-2}\\
	& \ge  \binom{x-1}{k-2} - \binom{y-1}{k-2} > 0
\end{align*}
contradicting the minimality of $w$. Hence $w$ starts with $0$ and similarly (by considering the reverse)
$w$ ends with $0$. We next show that $x-y=1$. Assume the contrary (that $x-y \ge 3$).
Let $w^*$ be obtained from $w$ by changing both $w_1$ and $w_n$ to $1$.
It is easily verified that $m(k,w)-m(k,w^*) =\binom{x-2}{k-2}-\binom{y}{k-2} > 0$, a contradiction.
So, we now know that $x-y=1$, that $w$ starts and ends with $0$ and is not alternating.
Hence, there must be some position $t$ such that $w_t=w_{t+1}=0$.
Now, there are two possible cases: (i) there exists a smallest position $r > t$ such that
$w_r=1$, $w_t=\cdots=w_{r-1}=0$ and $q^<(0,r-1,w) > q^<(1,r-1,w)$, or else (ii) there is a largest position $r < t$
such that $w_r=1$, $w_{r+1}=\cdots=w_t=0$ and $q^>(0,r+1,w) > q^>(1,r+1,w)$.
The two cases are symmetrical as if one of them does not hold, then it holds in the reverse of $w$. (For example, in $0111010001010$ we may choose $t=7$ since $w_7=w_8=0$. The first case does not hold since
although $w_{10}=1$, we have $q^<(0,9,w) = q^<(1,9,w) = 4$. The second case does hold since
$w_6=1$ and $q^>(0,7,w) > q^>(1,7,w)$. Notice that if we take the reverse $0101000101110$ then we can choose $t=5$ and the first case holds since we have $w_8=1$ and $q^<(0,7,w) > q^<(1,7,w)$.
As another example, in $01100$ the first case does not hold simply since $w_r$ does not exist, but in the reverse $00110$ it does hold as we $w_3=1$ and $q^<(0,2,w) > q^<(1,2,w)$).
Hence, assume without loss of generality that (i) holds, so $q^<(0,r-1,w) > q^<(1,r-1,w)$.
But notice that this also implies that $q^>(1,r,w) \ge q^>(0,r,w)$ since $x-y=1$.
Let $w^*$ be obtained from $w$ by flipping $w_r$ and $w_{r-1}$.
We have
$$
m(k,w)-m(k,w^*) =  \sum_{h=0}^{k-2}\binom{q^<(0,r-1,w)}{h}\binom{q^>(1,r,w)}{k-2-h}
\! - \! \binom{q^<(1,r-1,w)}{h}\binom{q^>(0,r,w)}{k-2-h}\;.
$$
But since $q^<(0,r-1,w) > q^<(1,r-1,w)$ and since $q^>(1,r,w) \ge q^>(0,r,w)$
we have that $m(k,w)-m(k,w^*) > 0$, a contradiction.
We have proved that if $n$ is odd, $w$ minimizing $m(k,w)$ must be alternating, as required.
\end{proof}

\subsection{Ternary alphabet}

\begin{proof}[Proof of Theorem \ref{t:exact-3}]
We assume that all words are over $\Sigma_3=\{0,1,2\}$.
Our proof consists of two stages. In the first stage we prove the the structure of an extremal word is of a certain form, and in the second stage we optimize over the parameters of that form.

Let $y$ be a given nonnegative integer. Let $n > 2y$ be an odd integer and consider all $n$-words over $\Sigma_3$
with $2y$ $1$'s and where the number of $0$'s, denoted $x$, is larger than the number of $2$'s, denoted $z$.
Under these restrictions, let $w$ be an $n$-word minimizing $m(3,w)$.
We will prove that $w$ must start with $y$ $1$'s, proceed with alternating $0$'s and $2$'s, and end with $y$ $1$'s. For example, if $y=2$ and $n=9$ then $w=110202011$. Let $u$ be the subword of $w$ obtained by ignoring all $1$'s.

\begin{lemma}\label{l:1}
	$u$ is alternating.
\end{lemma}
\begin{proof}
We first prove that $u$ must start and end with $0$. Assume, to the contrary, that $u$ starts with $2$ and let $t$ be the smallest position such that $w_t=0$ (so we have $t \ge 2$) and let $r < t$ be the largest position such that $w_r=2$. For example, if $w=122111000$ then $t=7$ and $r=3$. Notice that $w_{r+1}=\cdots=w_{t-1}=1$
(possibly $r+1=t$). Flip $w_r$ and $w_t$ to obtain the word $w^*$. For the last example, we have $w^*=120111200$.
To evaluate $m(3,w)-m(3,w^*)$ we need to consider $3$-words of $w$ that were affected by the flip.
These can be partitioned into three types: Those $3$-words that contain both locations $r,t$, those that contain
location $r$ but not $t$, and those that contain $t$ but not $r$ ($3$-words not of one of these types have not changed). Let $A$, $B$ and $C$ denote the respective contribution of each of these types to $m(3,w)-m(3,w^*)$.
We have:
\begin{align*}
	A & = q^>(0,t,w) + q^<(2,r,w) - q^>(2,t,w)\\
	 & \ge q^>(0,t,w) - q^>(2,t,w)\\
	 & \ge (x-1) - (z-1)\\
	 & > 0\;.
\end{align*}
As for the second type, we have
\begin{align*}
	B & = q^<(2,r,w)\cdot(t-r-1)+(t-r-1)\cdot (q^>(1,t,w)+q^>(0,t,w)) \\
	& - q^<(1,r,w)\cdot(t-r-1)-(t-r-1)\cdot q^>(2,t,w)
\end{align*}
and similarly for the third type we have
\begin{align*}
	C & = (q^<(2,r,w)+q^<(1,r,w))\cdot(t-r-1)+(t-r-1)\cdot q^>(0,t,w) \\
	& - (t-r-1)\cdot (q^>(1,t,w)+q^>(2,t,w))\;.
\end{align*}
As $x > z$ we have
\begin{align*}
B+C & = 2(t-r-1)(q^<(2,r,w) + q^>(0,t,w) - q^>(2,t,w)) \\
& \ge 2(t-r-1)(q^>(0,t,w) - q^>(2,t,w)) \\
& \ge 2(t-r-1)(x-1-z) \ge 0\;.
\end{align*}
Hence, $m(3,w)-m(3,w^*) = A+B+C > 0$, a contradiction. Similarly (by considering the reverse)
$u$ ends with $0$. We next show that $x-z=1$. Assume the contrary (that $x-z \ge 3$).
Let $w^*$ be obtained from $w$ by changing the first and last letters of $u$ from $0$ to $2$.
It is easily verified that $m(3,w)-m(3,w^*) \ge (x-2)-z > 0$, a contradiction to the minimality of $w$
(as $w^*$ still has $2y$ $1$'s).
So, we now know that $x-z=1$, that $u$ starts and ends with $0$ and assume it is not alternating
(otherwise we are done).
Hence, there must be some position $t$ such that $u_t=u_{t+1}=0$.
As in the proof of the binary case, there are now two possible cases:
(i) there exists a smallest position $r >t$ such that
$u_r=2$, $u_t=\cdots=u_{r-1}=0$ and $q^<(0,r-1,u) > q^<(2,r-1,u)$, or else (ii) there is a largest position $r < t$
such that $u_r=2$, $u_{r+1}=\cdots=u_t=0$ and $q^>(0,r+1,u) > q^>(2,r+1,u)$.
Notice that the two cases are symmetrical as if one of them does not hold, then it holds in the reverse of $w$
(which also reverses $u$).
Hence, assume without loss of generality that (i) holds, so $q^<(0,r-1,u) > q^<(2,r-1,u)$.
But notice that this also implies that $q^>(2,r,u) \ge q^>(0,r,u)$ since $x-z=1$.
Returning to $w$, position $r$ in $u$ corresponds to some position $r'$ in $w$ and position $r-1$ in $u$
corresponds to some position
$r''$ in $w$ where we have $r'' < r'$ and all letters strictly between $r''$ and $r'$ (if there are any)
are $1$'s. Let $w^*$ be obtained from $w$ by flipping $w_{r'}$ and $w_{r''}$.
We have that
\begin{align*}
m(3,w)-m(3,w^*) & = (2(r'-r'')-1) \cdot (q^<(0,r'',w)+q^>(2,r',w)-q^<(2,r'',w)-q^>(0,r',w))\\
 & \ge  q^<(0,r'',w)+q^>(2,r',w)-q^<(2,r'',w)-q^>(0,r',w) \\
 & = q^<(0,r-1,u)+q^>(2,r,u)-q^<(2,r-1,u)-q^>(0,r',u)\\
 & > 0\;,
\end{align*}
a contradiction. We have proved that $u$ is alternating, as claimed.
\end{proof}

\begin{lemma}\label{l:2}
	$u$ is a consecutive subword of $w$. Furthermore, $w$ begins with $y$ $1$'s and ends with $y$ $1$'s.
\end{lemma}
\begin{proof}
	Assume that $u$ is not consecutive, or that it is consecutive but $w$ does not begin with $y$ $1$'s nor
	end with $y$ $1$'s.
	Let $t$ be the smallest position such that $w_t=1$ and $w_{t-1} \neq 1$.
	We may assume that $q^<(1,t,w) < q^>(1,t,w)$ (recall that the total number of $1$'s is even) as
	otherwise we can use the reverse of $w$. Let $p$ be the smallest index such that $w_p \neq 1$.
	Let $w^*$ be obtained by moving $w_t$ to position $p$ and shifting $w_p \cdots w_{t-1}$ one position to the right. For example, if $w=111021110201201$ then $p=4$, $t=6$ and $w^*=111102110201201$.
	For notational clarity, let $a_j=q^<(j,t,w)$ and let $b_j=q^>(j,t,w)$.
	We have
$$
	m(3,w)-m(3,w^*) = a_0b_2+a_2b_0+b_1a_0+b_1a_2-a_0b_0-a_2b_2-a_1a_0-a_1a_2\;.
$$
But notice that $b_1=q^>(1,t,w) > a_1 = q^<(1,t,w)$ and that $a_0+a_2=t-p > 0$, so
$$
	m(3,w)-m(3,w^*) > a_0b_2+a_2b_0-a_0b_0-a_2b_2\;.
$$
But since $u$ is alternating, we have that either $a_0=a_2$ or that $b_0=b_2$. In either case, the r.h.s. of the last inequality is $0$, so $m(3,w)-m(3,w^*) > 0$, a contradiction.
\end{proof}

A word $w$ satisfying the statement of Lemma \ref{l:2} is said to be of {\em proper form}.
Namely, $w$ is of the form $1^yu1^y$ where $u=0202...20$ is alternating.
Let $n \ge 3$ be given and let $w^*$ be a word such that $m(3,w^*) \le m(3,w^{**})$ for all $w^{**} \in
(\Sigma_3)^n$. Equivalently, $f(3,3,n)=m(3,w^*)/\binom{n}{3}$.
If $n$ is odd and the number of $1$'s in $w^*$ is even then, by Lemma \ref{l:2}, we may assume that $w^*$ is of proper form. Otherwise, there exists a word $w$ of proper form of length $n'$ where $n-2 \le n' \le n$ such that
$f(3,3,n')=m(3,w)/\binom{n'}{3}$. Clearly if we append $n-n' \le 2$ arbitrary letters to $w$ we obtain an $n$-word. As any letter participates in $O(n^2)$ monotone $3$-words we have:
$$
f(3,3,n')+o_n(1)=\frac{m(3,w)+O(n^2)}{\binom{n'}{3}} \ge \frac{m(3,w^*)}{\binom{n'}{3}} \ge
\frac{m(3,w^*)}{\binom{n}{3}} = f(3,3,n)\;.
$$
Hence, in order to determine $f(3,3)$ it suffices to consider words of proper form.
For odd $n$ and $0 \le y < n/2$ let $w(n,y)=1^y0202...201^y$. So, our goal is to determine
$0 \le y < n/2$ such that $m(3,w(n,y))$ is minimized. Setting $y=\alpha n$ it is easily verified that
\begin{align*}
\frac{m(3,w(n,\alpha n))}{\binom{n}{3}} & = \frac{3}{4}(1-2\alpha)^3+(2\alpha)^3+6\alpha^2(1-2\alpha)+12\alpha(\tfrac{1}{2}-\alpha)^2 +o_n(1)\\
& = \frac{3}{4}-\frac{3\alpha}{2} + 3\alpha^2 + 2\alpha^3 + o_n(1)\;.
\end{align*}
Hence
\begin{equation}\label{e:h3}
f(3,3) = \min_{0 \le \alpha \le \frac{1}{2}} \frac{3}{4}-\frac{3\alpha}{2} + 3\alpha^2 + 2\alpha^3 = 2-\sqrt{2}
\end{equation}
where the minimum is obtained at $\alpha=\frac{\sqrt{2}-1}{2}$\,. 
\end{proof}

\section{A construction for every alphabet}\label{sec:construction}

We present an explicit construction of a sequence $S_s= \{w^n\}_{n=1}^{\infty}$ of words over $\Sigma_s$ where $w^n$ is an $n$-word, which is conjectured to be asymptotically optimal.
In other words, it is conjectured that $m(3,w^n)/\binom{n}{3}$ tends to $f(s,3)$ as $n$ goes to infinity. 
By {\em explicit} we mean that for every fixed $s$ and every given $n$, there is a polynomial time algorithm that generates $w^n$.

For two letters $\ell_1,\ell_2$, we say that a word is $(\ell_1,\ell_2)$-alternating if
it only contains the letters $\ell_1,\ell_2$ and any two consecutive letters are distinct.
To describe the construction, it is convenient to distinguish cases according to the parity of $s$.

Suppose first that $s$ is even, and let $x_1,\ldots x_{s/2-1}$ be real variables such that
$x_i \ge 0$ and $x_1+\cdots + x_{s/2-1} \le \frac{1}{2}$.
We say that an $n$-word over $\Sigma_s$ is of {\em folded form corresponding to $x_1,\ldots, x_{s/2-1}$}
if it is a palindrome and
the first (hence last) $\lfloor x_1n \rfloor$ letters are $(s/2-1,s/2)$-alternating,
the next $\lfloor x_2n \rfloor$ letters are $(s/2-2,s/2+1)$-alternating
and so on until the next $\lfloor x_{s/2-1}n \rfloor$ letters are $(1,s-2)$-alternating and finally
the remaining (hence central positions) are $(0,s-1)$-alternating.
Let $F_s(x_1,\ldots x_{s/2-1})$ be the set of all words of folded form corresponding to $x_1,\ldots, x_{s/2-1}$
and observe that for every $n \ge 1$ there is at least one $n$-word in $F_s(x_1,\ldots, x_{s/2-1})$.
For example, suppose $s=6$, $x_1=0.2$, $x_2=0.25$ and $n=11$, then $23145054132 \in F_6(0.2,0.25)$.
Clearly, for any two $n$-words $w,w'$ of $F_s(x_1,\ldots, x_{s/2-1})$, $|m(3,w)-m(3,w')|=O(n^2)$.

For odd $s$, our construction is just slightly different. Let $x_1,\ldots, x_{(s-1)/2}$ be real variables such that $x_i \ge 0$ and $x_1+\cdots + x_{(s-1)/2} \le \frac{1}{2}$.
We say that an $n$-word over $\Sigma_s$ is of {\em folded form corresponding to $x_1,\ldots, x_{(s-1)/2}$}
if it is a palindrome and
the first (hence last) $\lfloor x_1n \rfloor$ letters are all equal to $(s-1)/2$,
the next $\lfloor x_2n \rfloor$ letters are $((s-3)/2,(s+1)/2)$-alternating,
the next $\lfloor x_3n \rfloor$ letters are $((s-5)/2,(s+3)/2)$-alternating,
and so on until the next $\lfloor x_{(s-1)/2}n \rfloor$ letters are $(1,s-2)$-alternating and finally
the remaining (hence central positions) are $(0,s-1)$-alternating.
Let $F_s(x_1,\ldots, x_{(s-1)/2})$ be the set of all words of folded form corresponding to $x_1,\ldots, x_{(s-1)/2}$ and observe that for every $n \ge 1$ there is at least one $n$-word in $F_s(x_1,\ldots, x_{(s-1)/2})$.
For example, suppose $s=5$, $x_1=0.2$, $x_2=0.25$ and $n=11$, then $22310401322 \in F_5(0.2,0.25)$.
Clearly, for any two $n$-words $w,w'$ of $F_s(x_1,\ldots x_{(s-1)/2})$, $|m(3,w)-m(3,w')|=O(n^2)$.

Let $h_s(x_1,\ldots, x_{\lfloor (s-1)/2 \rfloor})$ be the {\em unique} polynomial such that
for every $w \in F_s(x_1,\ldots, x_{\lfloor (s-1)/2 \rfloor})$ it holds that 
$m(3,w)/\binom{n}{3} = h_s(x_1,\ldots, x_{\lfloor (s-1)/2 \rfloor}) + o_n(1)$ 
(notice that $h_s$ is indeed unique since recall that for any two $n$-words $w,w'$ of
$F_s(x_1,\ldots,x_{\lfloor (s-1)/2 \rfloor})$, $|m(3,w)-m(3,w')|=O(n^2)$). Clearly,
$h_s$ is a degree $3$ polynomial and clearly
\begin{proposition}\label{prop:fs3}
\[
\pushQED{\qed}
f(s,3) \le \min_{x_i \ge 0\;,\; x_1 + \cdots + x_{\lfloor (s-1)/2 \rfloor} \le \frac{1}{2}} 
h_s(x_1,\ldots, x_{\lfloor (s-1)/2 \rfloor})\;. \qedhere \popQED
\]
\end{proposition}
\noindent We conjecture that this upper bound is tight.
\begin{conjecture}\label{conj:1}
\begin{equation}\label{e:polymin}
	f(s,3) = \min_{x_i \ge 0\;,\; x_1 + \cdots + x_{\lfloor (s-1)/2 \rfloor} \le \frac{1}{2}} 
	h_s(x_1,\ldots, x_{\lfloor (s-1)/2 \rfloor})\;.
\end{equation}
\end{conjecture}
Notice that for any fixed $s$ we can generate the coefficients of  $h_s(x_1,\ldots, x_{\lfloor (s-1)/2 \rfloor})$
in constant time. Hence, given $n$, we can approximate the r.h.s. of \eqref{e:polymin}
and the corresponding optimal solution vector $(x_1,\ldots, x_{\lfloor (s-1)/2 \rfloor})$ to precision
$1/n$ in polynomial time. Consequently we may construct an $n$-word $w^n$ in $F_s(x_1,\ldots, x_{\lfloor (s-1)/2 \rfloor})$ in polynomial time, hence the claim regarding the explicit construction of $S_s$ in the beginning of this section.

\subsection{$h_s$ for $s=3,4,5,6,7$ and the upper bounds of Theorem \ref{t:almost-exact}}\label{subsec:upper}
Here we exhibit $h_s$ for some small $s$ and compute the corresponding minimum in \eqref{e:polymin}.
It is immediate to see that $h_3(x_1)=\frac{3}{4}-\frac{3x_1}{2} + 3{x_1}^2 + 2{x_1}^3$ is exactly the polynomial in \eqref{e:h3} and the solution to \eqref{e:polymin} is $2-\sqrt{2}$, as we have proved in Theorem \ref{t:exact-3} that the asymptotic solution for $s(3,3)$ can be obtained by considering a folded form corresponding to
$x_1=\frac{\sqrt{2}-1}{2}$. Equivalently, the proof of Theorem \ref{t:exact-3} regrading $f(3,3)$ shows that
Conjecture \ref{conj:1} holds for $s=3$.

For $s=4$ we have that $h_4(x_1)$ is a univariate polynomial.
One may manually compute $h_4(x)$ by considering the number of monotone $3$-word palindromes of the form $w=12..1203..3021..21$ where the first $(1,2)$-alternating block is of size $\lfloor xn \rfloor$.
To do this, we consider all possible monotone $3$-words of $\Sigma_4$ (of which there are $36$,
where four of which do not appear in $w$; these are $013$, $310$, $023$, $320$).
For each such word, we compute its density in $w$. For example, the density of $133$ is $\frac{3}{2}x(\frac{1}{2}-x)^2+o_n(1)$. Summing the corresponding density for all $32$ monotone $3$-words of $\Sigma_4$ that appear in $w$ we obtain
$$
h_4(x) = 9x^2(\tfrac{1}{2}-x)+12x(\tfrac{1}{2}-x)^2+6x^3+\frac{3}{4}(1-2x)^3 = 3x^3 + \frac{3}{2}x^2-\frac{3}{2}x + \frac{3}{4}\;.
$$
We obtain that $h_4(x)$ is minimized at $x=\frac{\sqrt{7}-1}{6}$
at which it is equal to $\frac{37-7\sqrt{7}}{36} \approx 0.5133...$.

For $s=5$ we have that $h_5(x_1,x_2)$ is a bivariate polynomial of degree $3$ and one can manually compute
$h_5(x)$ by considering the number of monotone $3$-word palindromes of the form
$w=2..213..1304...4031..312..2$ where the first block of $2$'s is of size
$\lfloor xn \rfloor$ and the next $(1,3)$-alternating block is of size $\lfloor yn \rfloor$.
It is easy (though lengthy) to verify that:
$$
h_5(x,y) = 2x^3+6x^2y+3x^2+9y^2x-\frac{3}{2}x+3y^3+\frac{3}{2}y^2-\frac{3}{2}y+\frac{3}{4}\;.
$$
We have
\begin{align*}
\frac{\partial h_5(x,y)}{\partial x} & = -\frac{3}{2} + 6x^2 + 9y^2 + 6x + 12xy\;,\\
\frac{\partial h_5(x,y)}{\partial y} & = -\frac{3}{2} + 6x^2 + 3y + 18xy + 9y^2\;.
\end{align*}
Standard calculus gives that the minimum of $h_5(x,y)$ where $x \ge 0$, $y \ge 0$ and
$x+y \le \frac{1}{2}$ is obtained at $(x,y)=(0.124772..., 0.199708...)$ where $x=0.124772...$ is the (unique) positive root of $16t^4+64t^3+56t^2-1$ and $y=0.199708...=1-\frac{1}{1+2x}$ is obtained from $x$ by equating any of the partial derivatives above to zero. At this minimum point we have $h_5(x,y)=0.4610...$.

For $s=6$ we obtain
$$
h_6(x,y) = 3x^3+3y^3+6x^2y+9xy^2+\frac{3}{2}x^2+\frac{3}{2}y^2-\frac{3}{2}x-\frac{3}{2}y+\frac{3}{4}\;.   
$$
We have
\begin{align*}
	\frac{\partial h_6(x,y)}{\partial x} & = 9x^2+12xy+3x+9y^2-\frac{3}{2}\;,\\
	\frac{\partial h_6(x,y)}{\partial y} & = 6x^2+18xy+3y+9y^2-\frac{3}{2}\;.
\end{align*}
Standard calculus gives that the minimum of $h_6(x,y)$ where $x \ge 0$, $y \ge 0$ and
$x+y \le \frac{1}{2}$ is obtained at $(x,y)=(0.189186..., 0.163220...)$ where $x=0.189186...$ is the (unique) positive root of $46t^4+68t^3+24t^2-2t-1$ and $y=0.163220...=\frac{x^2+x}{1+2x}$ is obtained from $x$ by equating any of the partial derivatives above to zero. At this minimum point we have $h_6(x,y)=0.428809...$.

For $s=7$ we obtain
\begin{align*}
h_7(x,y,z) & = \frac{3}{4} - \frac{3}{2}x + 3 x^2 + 2 x^3 - \frac{3}{2} y + 6 x^2 y + \frac{3}{2} y^2 + 9 x y^2 + 3 y^3\\
&  ~~~- \frac{3}{2} z + 6 x^2 z + 12 x y z + 6 y^2 z + \frac{3}{2} z^2 + 9 x z^2 + 9 y z^2 + 3 z^3\;.   
\end{align*}
We have
\begin{align*}
	\frac{\partial h_7(x,y,z)}{\partial x} & = -\frac{3}{2}+6x^2+9y^2+12yz+9z^2+6x+12xy+12xz\;,\\
	\frac{\partial h_7(x,y,z)}{\partial y} & = -\frac{3}{2}+6x^2+3y+18xy+9y^2+12xz+12yz+9z^2\;,\\
	\frac{\partial h_7(x,y,z)}{\partial z} & = -\frac{3}{2}+6x^2+6y^2+3z+18yz+9z^2+12xy+18xz\;.
\end{align*}
Standard calculus gives that the minimum of $h_7(x,y,z)$ where $x \ge 0$, $y \ge 0$, $z \ge 0$ and
$x+y+z \le \frac{1}{2}$ is obtained at $(x,y,z)=(0.0887976..., 0.150811...,0.135436...)$ where $x=0.0887976...$ is the (unique) positive root of $256t^8 + 2560 t^7 + 9088 t^6 + 14080 t^5 + 9248 t^4 + 1888 t^3 + 56 t^2 - 16 t - 1$ and $y=0.150811...=\frac{2x}{2x+1}$, $z=0.135436...=\frac{y^2+2x}{2x+2y+1}$ are obtained from $x$ by equating the partial derivatives above to zero. At this minimum point we have $h_7(x,y,z)=0.403383...$\:.

\section{Flag algebra and lower bounds for $f(s,k)$}\label{sec:flag}

The flag algebra method, introduced in a seminal paper of Razborov \cite{razborov-2007}, has become a
widely used and astonishingly effective method, mostly for homomorphism density problems in extremal combinatorics.
See \cite{razborov-2013} for a survey of flag algebra applications. To use this method in our setting,
we first need to introduce the combinatorial objects that we work with.

For an $n$-word $w$ over a totally ordered alphabet, the {\em word graph of $w$} is an edge-colored
undirected complete graph with vertex set $[n]$ and whose edges are colored by the function
$c:\binom{[n]}{2} \rightarrow \{0,1,2\}$ for which:
$$
c(i,j)= \begin{cases}
	0 & \text{if $w_i < w_j$}\,,\\
	1 & \text{if $w_i > w_j$}\,,\\
	2 & \text{if $w_i = w_j$}\,.
\end{cases}
$$
An (unlabeled) edge-colored complete graph is a {\em word graph} if it is the word graph of some word.
Notice that not every edge-colored complete graph with colors $\{0,1,2\}$ is a word graph and notice that the subset of word graphs not using the color $2$ is isomorphic to the well-known class of {\em permutation graphs}.
It is immediate to see that a $k$-subword of a word $w$ is monotone if and only if the $k$-clique corresponding to the subword in the word graph of $w$ does not use both the colors $0$ and $1$; we call such cliques {\em monotone}.
Hence, in order to study $f(s,k)$, it is equivalent to study the density of monotone $k$-cliques of word graphs of words over $\Sigma_s$. It is more effective to study the latter as there are significantly fewer word graphs than words. Let ${\mathcal G}(s,n)$ denote the set of all word graphs that correspond to $n$-words over $\Sigma_s$ and let ${\mathcal G}(s)=\cup_{n}{\mathcal G}(s,n)$. For example, the four elements of ${\mathcal G}(2,3)$ are depicted in Figure \ref{fig:f23} and
it is easy to verify that $|{\mathcal G}(s,3)|=8$ for all $s \ge 3$.
\begin{figure}
	\includegraphics[scale=0.6,trim=-20 370 340 72, clip]{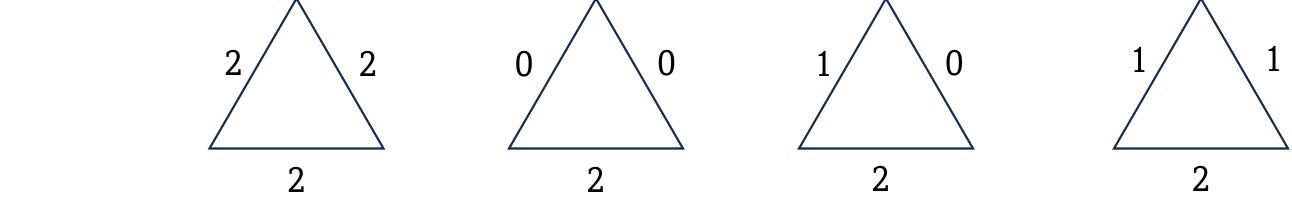}
	\caption{The elements of ${\mathcal G}(2,3)$.}
	\label{fig:f23}
\end{figure}

For two word graphs $H$ and $G$, let $P(H,G)$ denote the number of copies of $H$ in $G$ and let
$p(H,G)=P(H,G)/\binom{|V(G)|}{|V(H|}$ be the density of $H$ in $G$.
Let $m_k(G)$ denote the number of monotone $k$-cliques in $G$ and
notice that $m_k(G)$ equals $m(k,w)$ if $G$ is the word graph of $w$. 
Let $f_k(G)=m_k(G)/\binom{|V(G)|}{k}$ be the density of monotone $k$-cliques in $G$.
By double counting, we have for $G \in {\mathcal G}(s,n)$ that
\begin{equation}\label{e:dc}
	f_k(G) = \sum_{H \in {\mathcal G}(s,l)}f_k(H)p(H,G) ~~~\text{ for } k \le l \le n\;.
\end{equation}
 
We turn to our flag algebra notations and objects, which we describe in some brevity (yet in full), following an approach similar to \cite{BHL+-2015}.
The reader interested in more details may consult various surveys and gentle treatments to the subject, such as the one in \cite{GGHLM-2022}.

A {\em type} $\sigma$ is a bijectively vertex-labeled word graph from ${\mathcal G}(s,h)$ with label set $[h]$.
We call $h$ the {\em size} of $\sigma$ and,  for completeness, denote by $0$ the trivial empty type.
Notice that for all $s \ge 2$, there is a unique type of size $1$ and three distinct types of size $2$, corresponding to the three possible colorings of $K_2$ with colors $\{0,1,2\}$.

For a type $\sigma$ of size $h$, a {\em $\sigma$-flag} $F$ is a pair $(M, \theta)$ where $M$ is a word graph from ${\mathcal G}(s)$ and $\theta : [h] \rightarrow V(M)$ is an injective mapping such that $M[{\rm Im}(\theta)]$ is a copy of $\sigma$.
Two $\sigma$-flags $F=(M,\theta)$ and $F'=(M',\theta')$ are {\em flag-isomorphic} (denoted $F \cong F'$) if there is a graph isomorphism between $M$ and $M'$ (recall that $M$ and $M'$ are edge-colored so such an isomorphism should respect edge colors) which maps a labeled vertex in $F$ to a vertex with the same label in $F'$. Denote
by ${\mathcal F}_l^\sigma$ the set of $\sigma$-flags on $l$ vertices, up to flag isomorphism.
Note that ${\mathcal F}_l^0={\mathcal G}(s,l)$.

Given $\sigma$-flags $F \in {\mathcal F}_l^\sigma$ and $K=(M,\theta) \in {\mathcal F}_n^\sigma$, let
$p(F,K)$ be the probability that if we choose a random $l$-subset
$U$ of $V(M)$, subject to ${\rm Im}(\theta) \subseteq U$, then $(M[U], \theta) \cong F$.
(if $n < l$ or if $K$ is not a $\sigma$-flag, define $p(F,K)=0$).
Similarly, given flags $F,F' \in {\mathcal F}_l^\sigma$ and $K=(M,\theta) \in {\mathcal F}_n^\sigma$,
define the {\em joint density} $p(F,F';K)$ as the probability that if we choose two random $l$-subsets
$U,U'$ of $V(M)$, subject to $U \cap U' = {\rm Im}(\theta)$, then $(M[U], \theta) \cong F$ and $(M[U'], \theta) \cong F'$ (note: one can similarly define joint density when the flags $F,F'$ have different sizes, but we do not need this here; also, if $K$ is not a $\sigma$-flag or $n < 2l-|\sigma|$, define $p(F,F';K)=0$).
As an illustrative example of these notions, consider the type $\sigma$, and the $\sigma$-flags $F,F',K$ with the corresponding values of $p(F,K)$ and $p(F,F';K)$ given in Figure \ref{fig:flags}.

\begin{figure}
	\includegraphics[scale=0.59,trim=14 148 160 41, clip]{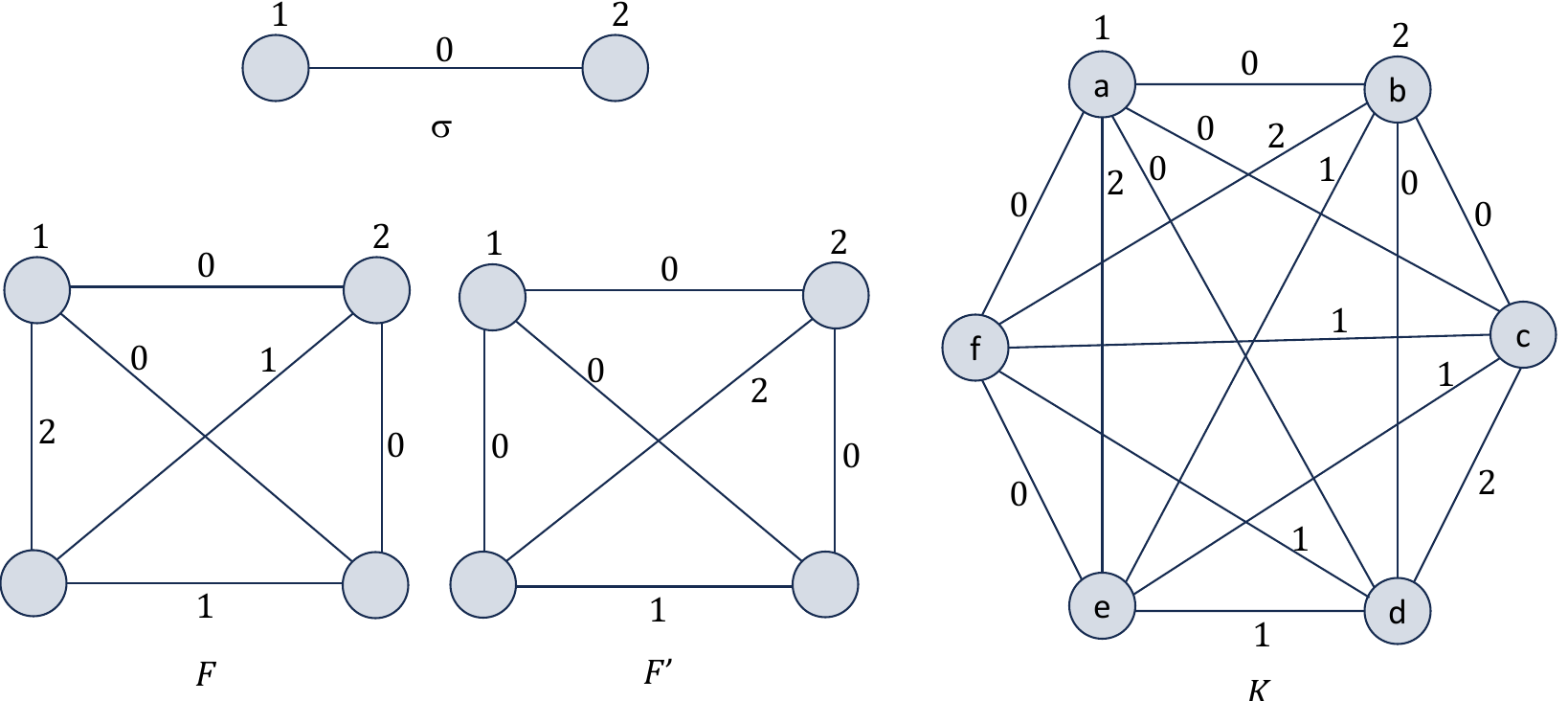}
	\caption{Depicted are a $\sigma$-type and three $\sigma$-flags $F,F' \in {\mathcal F}_4^\sigma$, $K=(M,\theta) \in {\mathcal F}_6^\sigma$ (notice that $M \in {\mathcal G}(3,6)$ as it is the word graph of, e.g., $012201$). It holds that $p(F,K)=\frac{2}{6}=\frac{1}{3}$ as only for the choices $U=\{a,b,d,e\}, \{a,b,c,e\}$ we have $(M[U],\theta) \cong F$. It also holds here that
	$p(F,F';K)=\frac{2}{6}$ as can be seen by the only possible choices $(U,U')=(\{a,b,d,e\},\{a,b,c,f\})$ and $(U,U')=(\{a,b,c,e\},\{a,b,d,f\})$ giving $(M[U], \theta) \cong F$ and $(M[U'], \theta) \cong F'$ .}
	\label{fig:flags}
\end{figure}

Now, suppose that ${\mathcal F}_{l^*}^\sigma = \{F_1,\ldots,F_t\}$ for some $l^* > |\sigma|$.
Let $Q$ be a $t \times t$ positive semidefinite matrix and let $K$ be some $\sigma$-flag with $n$ vertices.
It follows from positive-semidefiniteness and from Lemma 2.3 in \cite{razborov-2007} that:
$$
0 \le \sum_{i,j}Q[i,j]p(F_i,F_j;K)+o_n(1)\;.
$$
Notice that the last inequality holds also if $K$ is not a $\sigma$-flag as $p(F_i,F_j;K)=0$ in this case.
Let $\Theta(h,G)$ be the set of injective mappings from $[h]$ to $V(G)$ where $G \in {\mathcal G}(s,n)$. If we choose $\theta \in \Theta(|\sigma|,G)$ uniformly at random, then $K=(G,\theta)$ may or may not be a $\sigma$-flag.
In any case, we have from the last inequality and assuming that $l \ge 2l^*-|\sigma|$ that:
\begin{align*}
0 & \le \sum_{i,j}Q[i,j]{\mathbb E}_{\theta \in \Theta(|\sigma|,G)}[p(F_i,F_j;(G,\theta))]+o_n(1)\\
& = \sum_{H \in {\mathcal G}(s,l)}\left(\sum_{i,j}Q[i,j]{\mathbb E}_{\theta \in \Theta(|\sigma|,H)}[p(F_i,F_j;(H,\theta))]\right)p(H,G)+o_n(1)\;.
\end{align*}
Note that the coefficient of $p(H,G)$ is independent on $G$ as it only depends on
$\sigma,l^*,Q,H$ (recall that $s$ is fixed), so denote it by $c_H(\sigma,l^*,Q)$. We have
$$
	0 \le \sum_{H \in {\mathcal G}(s,l)}c_H(\sigma,l^*,Q)p(H,G)+o_n(1)\;.
$$
Let $(\sigma_i,l_i,Q_i)$ for $i \in [d]$ be a set of $d$ triples (where $l \ge 2l_i-|\sigma_i|$ for $1 \le i \le d$). For each triple, we consider the corresponding coefficient $c_H(\sigma_i,l_i,Q_i)$ and set $c_H = \sum_{i=1}^d c_H(\sigma_i,l_i,Q_i)$. Summing the last inequality for each triple we have:
$$
0 \le \sum_{H \in {\mathcal G}(s,l)}c_H \cdot p(H,G)+o_n(1)\;.
$$
The last inequality together with \eqref{e:dc} gives for $l \ge k$ that
$$
f_k(G) +o_n(1) \ge \sum_{H \in {\mathcal G}(s,l)}(f_k(H)-c_H)p(H,G) \ge \min_{H \in {\mathcal G}(s,l)}(f_k(H)-c_H)\;.
$$
We therefore obtain: 
\begin{corollary}\label{corr:flag}
	Let $s \ge 2$, $k \ge 3$ and $d \ge 1$ be integers. Let $l \ge k$ and let $(\sigma_i,l_i,Q_i)$ for
	$i \in [d]$
	be such that $\sigma_i$ is a type (whose underlying graph is from  ${\mathcal G}(s)$),
	$l_i > |\sigma_i|$ is an integer satisfying $l \ge 2l_i-|\sigma_i|$ and $Q_i$ is a positive semidefinite matrix indexed by ${\mathcal F}_{l_i}^{\sigma_i}$. Then,
	\[
	\pushQED{\qed}
	f(s,k) \ge \min_{H \in {\mathcal G}(s,l)}(f_k(H)-c_H)\;. \qedhere \popQED
	\] 
\end{corollary}
By Corollary \ref{corr:flag}, to obtain a lower bound for $f(s,k)$ we can choose $l \ge k$,
compute $f_k(H)$ for all $H \in {\mathcal G}(s,l)$, choose $d$ types $\sigma_i$ and $d$ sizes $l_i$ for
$i \in [d]$ and then solve a semidefinite program to
obtain solutions for the semidefinite matrices $Q_i$ for $i \in [d]$ that maximizes
$\min_{H \in {\mathcal G}(s,l)}(f_k(H)-c_H)$.
This turns out to be a lucrative avenue for $k=3$ and $s=4,5,6$ as detailed in the following subsection.

\subsection{The lower bounds of Theorem \ref{t:almost-exact}}\label{subsec:lb}

The results in this subsection reference a computer program which we call the ``generator program'' (link to code given in Table {\ref{table:urls}) and also reference Appendix \ref{appendix:A} which provides additional technical details.

Table \ref{table:gsl} lists the size of the word graph set ${\mathcal G}(s,l)$ for various $s,l$.
The generator program generates the sets ${\mathcal G}(s,l)$.
The values $|{\mathcal G}(s,8)|$ for $s \in \{5,6,7\}$ are still feasible when running the semidefinite program on a supercomputer but the added benefit in doing so, if any, is negligible (see Appendix A for more details).

\begin{table}[ht]
	\centering
	\begin{tabular}{c||c|c|c|c|c|c|c}
		&  $l=2$ & $l=3$ & $l=4$ & $l=5$ & $l=6$ & $l=7$ & $l=8$\\
		\hline
		$s=2$ & $3$ & $4$ & $10$ & $16$  & $36$   & $64$    & $136$  \\
		$s=3$ & $3$ & $8$ & $24$ & $76$  & $260$  & $848$   & $2760$ \\
		$s=4$ & $3$ & $8$ & $35$ & $146$ & $780$  & $3871$  & $18962$\\
		$s=5$ & $3$ & $8$ & $35$ & $179$ & $1248$ & $8978$  & $62394$\\
		$s=6$ & $3$ & $8$ & $35$ & $179$ & $1390$ & $12712$ & $119960$ \\
		$s=7$ & $3$ & $8$ & $35$ & $179$ & $1390$ & $13488$ & $155384$
	\end{tabular}
	\caption{The size of ${\mathcal G}(s,l)$ for various $s,l$.}
	\label{table:gsl} 
\end{table} 

The generator program computes $f_3(H)$ for every $H \in {\mathcal G}(s,l)$ which, recall, is the density of monotone triangles in $H$.

We shall use the types (valid for all $s \ge 3$) listed in Figure \ref{fig:types} as $\sigma_1,\ldots,\sigma_9$.
Hence, we have $d=9$ using the notation of the previous subsection.
\begin{figure}[ht]
	\includegraphics[scale=0.5,trim=10 380 20 10, clip]{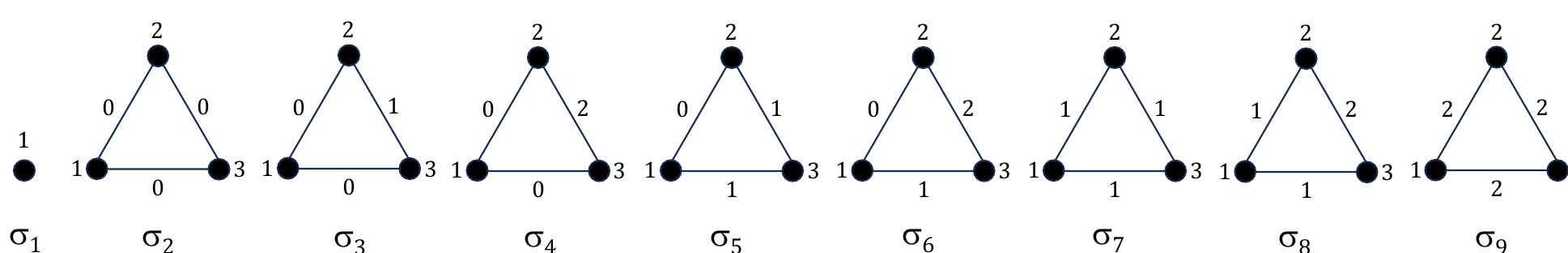}
	\caption{Types of size $3$ used in our program.}
	\label{fig:types}
\end{figure}

The generator program computes ${\mathcal F}_{l_i}^{\sigma_i}$. Specifically we shall use
$l_i=5$ for $2 \le i \le 9$ and use $l_1=4$.
Table \ref{table:flags} lists the sizes of these flag lists (note that some flag lists, hence their sizes, vary with the alphabet size $s$, as ${\mathcal G}(4,5) \subsetneq {\mathcal G}(5,5)={\mathcal G}(6,5) $).
\begin{table}[ht]
	\centering
	\begin{tabular}{c||c|c|c|c|c|c|c|c|c}
		&  $|{\mathcal F}_{4}^{\sigma_1}|$ &
		$|{\mathcal F}_{5}^{\sigma_2}|$ & $|{\mathcal F}_{5}^{\sigma_3}|$ & 
		$|{\mathcal F}_{5}^{\sigma_4}|$ & $|{\mathcal F}_{5}^{\sigma_5}|$ & 
		$|{\mathcal F}_{5}^{\sigma_6}|$ & $|{\mathcal F}_{5}^{\sigma_7}|$ & 
		$|{\mathcal F}_{5}^{\sigma_8}|$ & $|{\mathcal F}_{5}^{\sigma_9}|$ \\
		\hline
		$s=4$ & $80$ & $330$ & $305$ & $203$  & $305$   & $177$ & $330$ & $203$ & $110$ \\
		$s=5$ & $80$ & $402$ & $376$ & $203$  & $376$ & $177$ & $402$ & $203$ & $110$ \\
		$s=6$ & $80$ & $402$ & $376$ & $203$  & $376$ & $177$ & $402$ & $203$ & $110$
	\end{tabular}
	\caption{The size of ${\mathcal F}_{l_i}^{\sigma_i}$.}
	\label{table:flags} 
\end{table}

Finally, the generator program computes the last piece of data needed to generate the semidefinite program for a given $s \in \{4,5,6\}$, namely, for each pair of flags $F,F' \in {\mathcal F}_{l_i}^{\sigma_i}$ and for each
$H \in {\mathcal G}(s,l)$ it computes ${\mathbb E}_{\theta \in \Theta(|\sigma_i|,H)}[p(F,F';(H,\theta))]$,
where in our program we use $l=7$. Notice that $7$ is a valid choice for $l$ as recall that we must have 
$l \ge 2l_i - |\sigma_i|$.

Using the computed constants mentioned above, the generator program then creates for each $s \in \{4,5,6\}$, an input file in standard format (see Appendix \ref{appendix:A} for more details) that is fed to an sdp solver.
The solver gives for each $s \in \{4,5,6\}$ corresponding positive semidefinite matrices $Q_1,\ldots,Q_9$, for which we
obtain that
$\min_{H \in {\mathcal G}(4,7)}(f_3(H)-c_H) \approx 0.5123...$,
$\min_{H \in {\mathcal G}(5,7)}(f_3(H)-c_H) \approx 0.4604...$,
$\min_{H \in {\mathcal G}(6,7)}(f_3(H)-c_H) \approx 0.4280...$.
All of this bounds are proved rigorous using a rounding procedure elaborated upon in Appendix \ref{appendix:A}.
Hence, the lower bounds in Theorem \ref{t:almost-exact} hold as claimed.
Finally, note that together with the upper bounds given in Subsection \ref{subsec:upper}, Theorem \ref{t:almost-exact} is proved. \qed

\section{$f(s,k)$ versus $g(k)$}\label{sec:converge}

In this section we prove Theorem \ref{t:general}, which follows from the theorem and proposition below.
\begin{theorem}\label{t:general-lower}
	$f(s,k) \ge g(k)+\Theta(\frac{1}{s})$.
\end{theorem}
\begin{proof}
	Throughout the proof we shall assume that $k=\Theta(1)$ is fixed and $s$ grows.
	Let $q=(k-2)^2+1$ and assume that $n$ is divisible by $6q^2$.
	Let $w$ be an $n$-word over $\Sigma_s$. We say that a permutation $\pi \in S_n$ {\em respects} $w$  if $w_i < w_j$ implies $\pi(i) < \pi(j)$. Let $R$ be the set of all permutations that respect $w$.
	Notice that any monotone $k$-subword of $\pi$ corresponds to a monotone $k$-subword of $w$ for any $\pi \in R$, but the converse does not necessarily hold. For example, the permutation $213$ respects $w=001$ but
	the former is not monotone. By the definition of $g(k)$ we have that the density of monotone $k$-subwords of any permutation $\pi$ (in particular, for $\pi \in R$) is at least $g(k)-o_n(1)$. 
	
	Assume first that $w$ contains a letter $\ell$ that appears at least $n/6q^2$ times.
	Consider some subword of $w$ that contains only the letter $\ell$. While it is trivially monotone in $w$,
	the corresponding word in $\pi$ for a randomly chosen $\pi \in R$ is monotone with probability only $2/k!$, so the density of monotone subwords in $w$ is at least $g(k)-o_n(1)+\Theta(1) \ge g(k)-o_n(1)+\Theta(\frac{1}{s})$.
	Hence, we may and will assume that each letter appears at most $n/6q^2$ times in $w$.
	
	Partition $[n]$ into $q$ consecutive parts $V_1,\ldots,V_q$, each of size $n/q$.
	For $\pi \in R$, let $Z_i$ be the set of images of $\pi$ in $V_i$ (so,, e.g., $Z_1=\{\pi(1),\ldots,\pi(n/q)\}$).
	Consider another partition of $[n]$ into $6q^2$ consecutive parts $W_1,\ldots,W_{6q^2}$, each of size
	$n/6q^2$. We say that $W_j$ is {\em popular} in $Z_i$ if $|W_j \cap Z_i| \ge n/12q^3$.
	Notice that since $|Z_i|=n/q$ and since $|W_j|=n/6q^2$, there are at least $3q$ $W_j$'s that are popular in $Z_i$. We now perform a process of selecting one popular $W_j$ for each $Z_i$ as follows.
	Let $p_1$ be an index such that $W_{p_1}$ is popular in $Z_1$. Assume that we have already selected
	$p_1,p_2,\ldots, p_{m-1}$. We select $p_m$ such that $W_{p_m}$ is popular in $Z_m$ and
	$\{p_m-1, p_m, p_m+1 \} \cap \{p_1,\ldots,p_{m-1}\} = \emptyset$. Notice that we can complete this selection process for all $Z_1,\ldots,Z_q$ since there are at least $3q$ popular $W_j$'s for each $Z_i$.
	
	Consider the sequence $p_1,p_2,\ldots, p_q$. It is a set of $q$ distinct elements of $[6q^2]$ (and, furthermore, no two elements of this sequence differ by $1$). As $q=(k-2)^2+1$, we have  by the theorem of Erd\H{o}s and Szekeres \cite{ES-1935}, that it contains a monotone subsequence of length $k-1$, say $p_{j_1},\ldots, p_{j_{k-1}}$ and assume, without loss of generality, that it is monotone increasing. Hence $p_{j_r} < p_{j_{r+1}}$ and recall also that in fact $p_{j_r}+1 < p_{j_{r+1}}$ for $r=1,\ldots,k-2$.
	
	Now consider a randomly chosen $k$-subword of $w$, say $u=w_{x_0}\ldots,w_{x_{k-1}}$.
	What is the probability of the event $A$ that $\pi(x_r) \in W_{p_{j_r}} \cap Z_{j_r}$ for $1 \le r \le k-1$ and that $\pi(x_0) \in W_{p_{j_1}} \cap Z_{j_1}$? It is clearly $\Theta(1)$ since $|W_{p_{j_r}} \cap Z_{j_r}| \ge n/12q^3 = \Theta(n)$ (note that this does not depend on $s$). What is the probability of the event $B \subseteq A$ that it further holds that $w_{x_0}=w_{x_1}$?  Since $|W_{p_{j_1}} \cap Z_{j_1}| \ge n/12q^3$ and there are only $s$ letters, we have that $B$ holds with probability $\Theta(\frac{1}{s})$.
	Hence, a $\Theta(\frac{1}{s})$ fraction of the $k$-subwords of $w$ satisfy event $B$.
	We claim that if $u$ satisfies $B$ then it is monotone in $w$. Indeed, 
	this holds since $w_{x_0}=w_{x_1}$ and since $\pi(x_1)\ldots\pi(x_{k-1})$ is monotone increasing in $\pi$.
	To see the latter notice that $\pi(x_r) \in W_{p_{j_r}}$ and $\pi(x_{r+1}) \in W_{p_{j_{r+1}}}$
	and recall that $p_{j_r}+1 < p_{j_{r+1}}$ and that no letter appears more than $n/6q^2 =|W_{p_{j_r}+1}|$ times in $w$.
	But now, suppose that we choose $\pi \in R$ at random. Since $w_{x_0}=w_{x_1}$, the probability that
	$\pi(x_0) < \pi(x_1)$ is $\frac{1}{2}$. Hence, there exists $\pi \in R$ for which at most half of the
	$k$-subwords of $w$ that satisfy $B$ (and recall that all of which are monotone in $w$) correspond to monotone
	words in $\pi$. Hence, a $\Theta(\frac{1}{s})$ fraction of the $k$-subwords of $w$ are monotone in $w$ but not monotone in $\pi$. Therefore,
	the density of monotone $k$-subwords in $w$ is at least $g(k)-o_n(1)+\Theta(\frac{1}{s})$.
\end{proof}

\begin{proposition}\label{prop:general-upper}
	 $f(s,k) \le g(k)+\frac{\binom{k}{2}}{s} + \Theta\left(\frac{1}{s^2} \right)$.
\end{proposition}
\begin{proof}
	We shall assume that $s$ divides $n$ and construct an $n$-word with a small amount of monotone $k$-subwords.
	Let $\pi$ be a permutation of $[n]$ minimizing the number of monotone subsequences of length $k$.
	Let this number be $m(k,\pi)$, so by definition of $g(k)$ we have $m(k,\pi) = g(k)\binom{n}{k}(1+o_n(1))$.
	Let $V_\ell = \{\pi^{-1}(v)\;|\; \ell n/s < v \le (\ell+1)n/s\}$ for $0 \le \ell < s$. Construct an $n$- word $w$ over $\Sigma_s$ where $w_i=\ell$ if $i \in V_\ell$. We compare $m(k,w)$ to $m(k,\pi)$.
	Clearly, if a subsequence of $\pi$ is monotone, then the corresponding subsequence of $w$ is also monotone.
	The only way a non-monotone subsequence of $\pi$ can become monotone in $w$ is if the subsequence contains
	two locations $i$ and $j$ such that $i,j \in V_\ell$ for some $\ell$. The amount of such sequences is
	$$
	(1+o_n(1))\binom{n}{k}\left(1-\prod_{i=1}^{k-1}\left(1-\frac{i}{s}\right)\right) = (1+o_n(1))\binom{n}{k}\left(\frac{\binom{k}{2}}{s}+\Theta\left(\frac{1}{s^2} \right)\right)\;.
	$$
	It follows that $f(s,k) \le g(k)+\frac{\binom{k}{2}}{s} + \Theta\left(\frac{1}{s^2} \right)$.
\end{proof}

From Theorem \ref{t:general-lower} and Proposition \ref{prop:general-upper} we obtain
$f(s,k)-g(k)=\Theta(\frac{1}{s})$, proving Theorem \ref{t:general}. \qed

\newpage

\appendix

\section{Notes on the semidefinite programs}\label{appendix:A}

\subsection{Computational resources and limitations}

The state of the art numerical sdp solvers are CSDP \cite{borchers-1999} and SDPA \cite{YFK-2003}.
They are similar in their performance, accuracy, adhere to the same standard input format, and implement a primal-dual interior-point method. 
While efficient and accurate, being an interior-point method means, in particular, that they need to store and manipulate dense matrices whose orders are at least the number of constraints of the problem instance.
For example, CSDP, which we have used, requires at least $8m^2$ bytes of storage
\cite{borchers-1999}, where $m$ is the number of constraints, which, in
our problem, equals $|{\mathcal G}(s,l)|$. Constraint sizes of the order about $100000$ already require a supercomputer and constraint sizes in the millions are impractical using conventional hardware;
see \cite{LP-2021} for more details on hardware limitations when using CSDP.
To obtain meaningful lower bounds for $f(s,k)$ when $s \ge 7$ or $k \ge 4$ require using $|{\mathcal G}(s,l)|$
where $l \ge 8$ and even larger, which, by Table \ref{table:gsl} and the above discussion, becomes too computationally demanding.

Another obvious, yet crucial point, is that CSDP and SDPA use floating point arithmetic, and hence their results, while highly accurate, are, nevertheless, approximations. One then needs to apply some further rounding procedure
in order to turn the results into a rigorous proof. See Subsection \ref{subsec:rounding} for details on our rounding method of choice.

\subsection{Converting to standard sdp format}

CSDP solves the following standard form semidefinite program.
Let $C,A_1,\ldots,A_m$ be given real symmetric matrices and let $X$ be a (variable) real symmetric matrix.
Let $a=(a_1,\ldots,a_m)$ be a given real vector.
The semidefinite program solved by CSDP is:
\begin{equation}\label{e:csdp}
	\begin{aligned}
		\max \hspace{32pt} & \operatorname{tr}(C X) \\
		\text{subject to } & \operatorname{tr}(A_jX) = a_j & \text{for } 1 \le j \le m\;,\\
		& X \succeq 0
	\end{aligned}
\end{equation}
 (here $X \succeq 0$ means $X$ is positive semidefinite).

Recalling Corollary \ref{corr:flag} and the objects (i.e., types, flag lists, joint densities) of Subsection \ref{subsec:lb}, note that our semidefinite program, for fixed $s \in \{4,5,6\}$  is:
$$
\max_{Q_1,\ldots,Q_9} \min_{H \in {\mathcal G}(s,7)}(f_3(H)-c_H)
$$
subject to $Q_i \succeq 0$ and to $c_H = \sum_{i=1}^9 c_H(\sigma_i,l_i,Q_i)$.
Denoting ${\mathcal F}_{l_i}^{\sigma_i}=\{F^i_1,\ldots,F^i_{t_i}\}$ recall 
the definition of $c_H(\sigma_i,l_i,Q_i)$ given in Section \ref{sec:flag}:
$$
c_H(\sigma_i,l_i,Q_i) = \sum_{u=1}^{t_i}\sum_{v=1}^{t_i}
Q_i[u,v]{\mathbb E}_{\theta \in \Theta(|\sigma_i|,H)}[p(F^i_u,F^i_v;(H,\theta))]\;.
$$
Our program is therefore
$$
\max_{Q_1,\ldots,Q_9} \min_{H \in {\mathcal G}(s,7)}(f_3(H)-\sum_{i=1}^9 \sum_{u=1}^{t_i}\sum_{v=1}^{t_i}
Q_i[u,v]{\mathbb E}_{\theta \in \Theta(|\sigma_i|,H)}[p(F^i_u,F^i_v;(H,\theta))])
$$
subject to $Q_i \succeq 0$. Now, observe that if $M$ and $N_1,\ldots,N_9$ are positive integers then
the last program yields the same result as the program
$$
\frac{1}{M}\max_{Q_1,\ldots,Q_9} \min_{H \in {\mathcal G}(s,7)}(M \cdot f_3(H)-\sum_{i=1}^9 \sum_{u=1}^{t_i}\sum_{v=1}^{t_i}
Q_i[u,v]N_i{\mathbb E}_{\theta \in \Theta(|\sigma_i|,H)}[p(F^i_u,F^i_v;(H,\theta))])\;.
$$
The advantage of using the latter formulation is that we can choose $M$, $N_1,\ldots,N_9$ such that all constants in the
sdp are integers, which makes the input to the sdp more concise.
Since $f_3(H)$ is a rational with denominator $\binom{7}{3}$, we shall choose $M=35$.
Also notice that for $i=2,\ldots,9$, $p(F^i_u,F^i_v;(H,\theta))]$ is a rational with denominator
$1260$ (the number of choices for $\theta \in \Theta(|\sigma_i|,H)$ is $7\cdot 6 \cdot 5$ and the number of choices for a subset pair $(U,U')$ in the definition of joint density is $\binom{4}{2}$ in our case), so we choose $N_i=1260$. Analogously, 
$p(F^1_u,F^1_v;(H,\theta))]$ is a rational with denominator $140=7 \cdot \binom{6}{3}$ so we choose
$N_1=140$. We then execute the sdp and obtain by Corollary \ref{corr:flag} that
\begin{equation}\label{e:sdp}
35f(s,3) \ge \max_{Q_1,\ldots,Q_9} \min_{H \in {\mathcal G}(s,7)}(35f_3(H)-\sum_{i=1}^9 \sum_{u=1}^{t_i}\sum_{v=1}^{t_i}
Q_i[u,v]N_i{\mathbb E}_{\theta \in \Theta(|\sigma_i|,H)}[p(F^i_u,F^i_v;(H, \theta))])\;.
\end{equation}

Translating the sdp \eqref{e:sdp} to standard sdp format as in \eqref{e:csdp} is a straightforward process of adding slack variables. 
In our case $m=|{\mathcal G}(s,7)|$ is the dimension of the vector $a=(a_1,\ldots,a_m)$.
Now, suppose that ${\mathcal G}(s,7)=\{H_1,\ldots,H_m\}$ then $a_i=35f_3(H_i)$.
Each of the matrices $C,A_1,\ldots,A_m$ is a block diagonal matrix with precisely $10$ blocks, where
for $1 \le i \le 9$, the $i$'th block corresponds to the type $\sigma_i$ and the last block is the ``slack block''.
For $i=1,\ldots,9$, the order of block $i$ in each of these matrices is $|{\mathcal F}_{l_i}^{\sigma_i}|$,
namely it is the corresponding row in Table \ref{table:flags}. The dimension of the slack block is $m+1$.
As for their entries, the matrix $C$ is entirely zero except for the $[1,1]$ entry of the slack block which is $1$.
For $1 \le i \le 9$ and for $1 \le j \le m$, entry $[u,v]$ of block $i$ of $A_j$ is
$N_i{\mathbb E}_{\theta \in \Theta(|\sigma_i|,H_j)}[p(F^i_u,F^i_v;(H_j,\theta))]$.
The slack block of $A_j$ is entirely zero except for entries $[1,1]$ and $[j+1,j+1]$ which are $1$
(so the slack block is a diagonal matrix).
The input files referenced in Table \ref{table:urls} contain all of these values in standard
SDPA sparse format (see the manuals of either SDPA or CSDP for a description of this format, used by both programs).

\subsection{Rounding}\label{subsec:rounding}

As the output matrix of the sdp solver is an approximate floating point solution, one needs to couple the approximate result with a rounding argument in order to obtain a rigorous proof. There are several approaches addressing this task, broadly falling into two categories, depending on whether the solution of the sdp is claimed to be an exact result or a bound for the exact result.
In the ``exact'' case, one needs to convert the entries of the output matrix to rationals which can then
be certified, with exact arithmetic, to yield the claimed exact solution. These methods involve (sometimes ingenious \cite{grzesik-2012,HHKNR-2013}) ``guessing'' of the entires to close rationals, and are usually suited to problems containing a moderate amount of variables (by {\em variables} we mean the number of distinct entry positions of the resulting matrix that are potentially nonzero). In the ``bound'' case,
rounding is less stringent and hence suitable for sdp's with a large number of variables (see, e.g., \cite{ANS-2016,SS-2018}).
As in our problem we only claim a lower bound for $f(s,3)$, the rounding procedure we present is of the ``bound'' category; specifically, we will adopt to our setting the method from \cite{SS-2018}.

Let $X'$ denote the result matrix of the sdp execution.
Notice that $X'$ has the same block structure as the matrices $A_j$. Recalling that
block $i$ is square symmetric of order $|{\mathcal F}_{l_i}^{\sigma_i}|$ for $1 \le i \le 9$
and that the slack block is diagonal of order $m+1=|{\mathcal G}(s,7)| + 1$, we have that the number of variables in our problem is:
$$
	\sum_{i=1}^9 \binom{|{\mathcal F}_{l_i}^{\sigma_i}|+1}{2} + |{\mathcal G}(s,7)| + 1 = \begin{cases}
		272942 & \text{if $s=4$}\:,\\
		379247 & \text{if $s=5$}\:,\\
		382981 &  \text{if $s=6$}\:.
	\end{cases} 
$$

Let $\delta > 0$ be a parameter, which can be configured by CSDP and is taken by default to be $10^{-8}$
\cite{borchers-1999}. CSDP guarantees \cite{borchers-1999} that $X'$ satisfies:
\begin{equation}\label{e:guarantee}
	\begin{aligned}
		|\operatorname{tr}(A_jX') - a_j| & \le \delta & \text{for } 1 \le j \le m \;,\\
		X' & \succeq 0\;.
	\end{aligned}
\end{equation}
Nevertheless, as $X'$ consists of floating point numbers, we need to convert them to rationals in order to certify a rigorous proof, without incurring significant loss in the obtained bound.
Since $X' \succeq 0$, we use python numpy (which uses BLAS/LAPACK \cite{anderson-1999})  to compute the
Cholesky decomposition of $X'$, so $L'L'^T=X'$ where $L'$ is lower triangular. Nevertheless, the computed $L'$ is still a floating point approximation (since $X'$ is an approximation and since the Cholesky decomposition routine may further incur additional loss). Let $D$ be a large integer (specifically, we use $D=10^6$).
We multiply each entry of $L'$ by $D$ and round each resulting entry to the closest integer, thus we obtain
an integer matrix $L$. We expect that $\frac{1}{D^2}LL^T$ is a good approximation of the exact result $X$
of our sdp. The matrix $L$ computed by our python script (see Table \ref{table:urls}) is our certificate and is provided by the links in Table \ref{table:urls} as well.

Our final task is to verify that $L$ produces a bound close to optimal. Let $M=LL^T$. We first verify that $L$ is indeed the Cholesky decomposition of $M$. Since $L$ is lower triangular, this means that we just need
to verify that all diagonal entries of $L$ are positive. Indeed, as our script shows, this holds
(for $s=4,5,6$ the lowest diagonal entry of the corresponding $L$ is respectively, $31,15,4$).

Let $b_j = \operatorname{tr}(\frac{1}{D^2}A_jM-a_j)$. Since $A_j,M,a_j$ are all integral, we have that $b_j$
is a rational with denominator $D^2$. Let $\varepsilon = \max_{j=1}^m|b_j|$. Our program shows that in all cases
$s=4,5,6$ the obtained (rational) $\varepsilon$ is smaller than $0.0002$.


Notice that $\operatorname{tr}(\frac{1}{D^2}CM)$ is a lower bound for the {\em precise solution} of the problem
\begin{equation}\label{e:csdp-precise}
	\begin{aligned}
		\max \hspace{32pt} & \operatorname{tr}(CX) \\
		\text{subject to } & \operatorname{tr}(A_jX) = a_j+b_j & \text{for } 1 \le j \le m\;,\\
		& X \succeq 0\;.
	\end{aligned}
\end{equation}
Recalling that $a_j=35f_3(H_j)$ we have, in turn, that the precise solution of \eqref{e:csdp-precise} is the precise solution of 
$$
	\max_{Q_1,\ldots,Q_9} \min_{H_j \in {\mathcal G}(s,7)}(35f_3(H_j)+b_j-\sum_{i=1}^9 \sum_{u=1}^{t_i}\sum_{v=1}^{t_i}
	Q_i[u,v]N_i{\mathbb E}_{\theta \in \Theta(|\sigma_i|,H)}[p(F^i_u,F^i_v;(H, \theta))])\;.
$$
Hence, if $Q_1,\ldots,Q_9$ denote the non-slack blocks of $\frac{1}{D^2}M$, we have that
$$
\operatorname{tr}(\frac{1}{D^2}CM) \le \min_{H_j \in {\mathcal G}(s,7)}(35f_3(H_j)+b_j-\sum_{i=1}^9 \sum_{u=1}^{t_i}\sum_{v=1}^{t_i}
Q_i[u,v]N_i{\mathbb E}_{\theta \in \Theta(|\sigma_i|,H)}[p(F^i_u,F^i_v;(H, \theta))])\;.
$$
Since $|b_j| \le \varepsilon$, we have that
$$
\operatorname{tr}(\frac{1}{D^2}CM) - \varepsilon \le \min_{H_j \in {\mathcal G}(s,7)}(35f_3(H_j)-\sum_{i=1}^9 \sum_{u=1}^{t_i}\sum_{v=1}^{t_i}
Q_i[u,v]N_i{\mathbb E}_{\theta \in \Theta(|\sigma_i|,H)}[p(F^i_u,F^i_v;(H, \theta))])\;.
$$
Hence, by \eqref{e:sdp}, $\operatorname{tr}(\frac{1}{D^2}CM)-\varepsilon \le 35f(s,3)$.
Finally, running our python script referenced in Table \ref{table:urls}, or just examining entry
$[1,1]$ of the slack block of $M$ (recall that the only nonzero entry of $C$ is the $[1,1]$ entry of the slack block, which equals $1$), we have
$$
\operatorname{tr}(CM) = \begin{cases}
	17931108816196 & \text{if $s=4$}\:,\\
	16117334329600  & \text{if $s=5$}\:,\\
	14982659113536  &  \text{if $s=6$}\:.
\end{cases} 
$$
Recalling that $\varepsilon < 0.0002$ and $D=10^6$, the lower bounds in Theorem \ref{t:almost-exact} hold.
\begin{table}[ht]
	\centering
	\begin{tabular}{l|l}
		\hline
		Generator program & \footnotesize{\url{github.com/raphaelyuster/monotone-words/blob/main/sdp-generator.cpp}} \\
		SDP rounding script & \footnotesize{\url{github.com/raphaelyuster/monotone-words/blob/main/sdp_analyzer.py}} \\
		SDP input, $s=4$ & \footnotesize{\url{raw.githubusercontent.com/raphaelyuster/monotone-words/main/words4.dat-s}} \\
		SDP input, $s=5$ & \footnotesize{\url{raw.githubusercontent.com/raphaelyuster/monotone-words/main/words5.dat-s}} \\
		SDP input, $s=6$ & \footnotesize{\url{raw.githubusercontent.com/raphaelyuster/monotone-words/main/words6.dat-s}} \\
		Certificate, $s=4$ & \footnotesize{\url{github.com/raphaelyuster/monotone-words/blob/main/words4.cert}} \\
		Certificate, $s=5$ & \footnotesize{\url{github.com/raphaelyuster/monotone-words/blob/main/words5.cert}} \\
		Certificate, $s=6$ & \footnotesize{\url{github.com/raphaelyuster/monotone-words/blob/main/words6.cert}} \\
		Script output, $s=4$ & \footnotesize{\url{github.com/raphaelyuster/monotone-words/blob/main/words4.txt}} \\
		Script output, $s=5$ & \footnotesize{\url{github.com/raphaelyuster/monotone-words/blob/main/words5.txt}} \\
		Script output, $s=6$ & \footnotesize{\url{github.com/raphaelyuster/monotone-words/blob/main/words6.txt}} \\
		\hline
	\end{tabular}
	\caption{Description of each certificate and its corresponding url.}
	\label{table:urls} 
\end{table} 
\end{document}